\documentclass[11pt,reqno]{amsproc}
\usepackage{amsmath,amsthm,amsfonts}

\newcommand{\norm}[2]{\left\| {#1}\right\| _{#2}}
\newcommand{\abs}[1]{\left | {#1}\right | }

\newcommand{\all}[2]{\left \{\,{#1}\,:\,{#2}\,\right \}}

\newcommand{\dual}[2]{\left \langle {#1},{#2} \right \rangle}
\newcommand{\inn}[2]{\left ({#1},{#2} \right )}

\newcommand{\set}[1]{{\{#1\}}}

\newcommand{\HH}{\mathcal H}
\newcommand{\Si}{S_\infty}

\newcommand{\N}{\ensuremath{\mathbb N}}

\newcommand{\C}{\ensuremath{\mathbb C}}
\newcommand{\lam}{\lambda}
\newcommand{\open}{\mathcal O_d}

\theoremstyle{plain} \newtheorem{theorem}{Theorem}[section]
\newtheorem{prop}[theorem]{Proposition}
\newtheorem{lem}[theorem]{Lemma} 

\theoremstyle{definition} \newtheorem{defn}[theorem]{Definition}
\newtheorem{rem}[theorem]{Remark}
\newtheorem{notation}[theorem]{Notation}

\newtheoremstyle{citing}% name
  {}%      Space above, empty = `usual value'
  {}%      Space below
  {\itshape}% Body font
  {\parindent}%  Indent amount (empty = no indent, \parindent = para indent)
  {\scshape}% Thm head font
  {.}%        Punctuation after thm head
  { }%     Space after thm head: " " = normal interword space;
  {\thmnote{#3}}% Thm head spec

\theoremstyle{citing}
\newtheorem*{procthm}{}

\newcommand{\trop}{{\mathcal  L}}
\newcommand{\cc}{\subset\!\!\!\subset}

\newcommand{\I}{\mathcal I}
\newcommand{\J}{\mathcal J}

\renewcommand{\H}{\mathcal H}
\renewcommand{\l}{\mathcal L}
\renewcommand{\L}{\mathcal L}

\begin{document}

\title[Explicit eigenvalue estimates for transfer operators]{Explicit eigenvalue estimates
  for transfer operators acting
  on spaces of holomorphic functions}
\author{Oscar F.~Bandtlow and Oliver Jenkinson}

\address{Oscar F.~Bandtlow;
School of Mathematical Sciences, Queen
  Mary, University of London, Mile End Road, London, E1 4NS, UK.
\newline
{\tt ob@maths.qmul.ac.uk} \newline {\tt
    www.maths.qmul.ac.uk/$\sim$ob}}

\address{Oliver Jenkinson;
School of Mathematical Sciences, Queen
  Mary, University of London, Mile End Road, London, E1 4NS, UK.
  \newline {\tt omj@maths.qmul.ac.uk} \newline {\tt
    www.maths.qmul.ac.uk/$\sim$omj}}

\date\today\maketitle

\begin{abstract}
We consider
transfer
operators acting on spaces of holomorphic functions, and provide
explicit bounds for their eigenvalues.
More precisely, if $\Omega$ is any open set in $\C^d$, and $\l$
is a suitable transfer operator acting on Bergman
space $A^2(\Omega)$, its eigenvalue sequence
$\{\lambda_n(\l)\}$
is bounded by
$|\lambda_n(\l)|\le A\exp(-a n^{1/d})$, where
$a, A>0$ are explicitly given.
\end{abstract}

\section{Introduction}
\label{introsection}
The study of transfer operators acting on spaces of holomorphic
functions was initiated by Ruelle \cite{ruelleinventiones} in 1976.
He showed that certain dynamical zeta functions, including those
of Artin-Mazur \cite{artinmazur} and Smale \cite{smale}, could be
expressed in terms of the determinant of such operators.  The setting
for Ruelle's theory is (the complexification of) a real analytic
expanding map.  If $(\phi_i)_{i\in\I}$ are the local inverse
branches of this map, and $(w_i)_{i\in\I}$ is a suitable collection
of holomorphic functions, then the associated \emph{transfer operator}
$\L$, defined by
\begin{equation}
\label{introtransferopdefn}
(\L f)(z)
=
\sum_{i\in\I}
w_i(z) f(\phi_i (z))\,,
\end{equation}
preserves the space of functions holomorphic on some appropriate open
subset $\Omega$ of $d$-dimensional complex Euclidean space.

Transfer operators of this form arise in statistical mechanics (see
\cite{ruellestatmech}), and have been applied to hyperbolic dynamical
systems, notably by Ruelle \cite{ruellecmp68},
Sinai \cite{sinai72}, and Bowen \cite{bowenbook}, as part of
their program
of thermodynamic formalism (cf.~\cite{ruellebook}).
Up until 1976 the setting for this formalism was symbolic dynamics: a
hyperbolic system can be coded by a subshift of finite type $\Sigma$,
and the transfer operator $\L$ preserves the space of Lipschitz
functions on $\Sigma$.  If the functions $w_i$ are positive then $\L$
inherits a positivity property, and an infinite dimensional analogue
of the Perron-Frobenius theorem can be established
(cf.~\cite{ruellecmp68}): the leading eigenvalue of $\L$ is simple,
positive, and isolated.  This leads to important ergodic-theoretic
information (e.g.~exponential decay of correlations) about a wide
class of invariant measures (equilibrium states).  Variations on this
result have continued to be a fruitful area of active development (see
\cite{baladibook} for a comprehensive overview), with transfer
operators studied on various other spaces, notably $C^k$ spaces
\cite{ruellecmp,ruellepublihes}, and the space
of functions of
bounded variation \cite{lasotayorke, hofbauerkeller, boyarskygora}.
In each of these
cases $\L$, although not a compact operator, does enjoy the
Perron-Frobenius property of having an isolated and positive dominant
eigenvalue.  In the case where $\L$ acts on certain \emph{holomorphic}
function spaces, however, Ruelle \cite{ruelleinventiones} showed that
it enjoys much stronger properties. In particular $\L$ is compact, so
that its spectrum is a sequence $\{\lambda_n(\L)\}$ converging to
zero, together with zero itself.

The present article is concerned with obtaining completely explicit
upper bounds on the eigenvalue moduli $|\lambda_n(\L)|$, ordered by
decreasing modulus and counting algebraic multiplicities. Spectral
estimates of this kind have a long history (see
e.g.~\cite[Ch.~7]{pietsch2}),
and the theory is particularly well developed
in the case where $\L$ is the Laplacian, or more generally a
selfadjoint differential operator.  Relatively little is known in the
non-selfadjoint case, however, and existing explicit bounds on the
eigenvalues of transfer operators are mainly restricted to the first
two eigenvalues, where positivity arguments can be employed.

Explicit information on the spectrum of transfer operators is
desirable for a variety of reasons.  For example any explicit estimate
on the second eigenvalue $\lambda_2(\L)$ yields an explicit bound on
the exponential rate of mixing for the underlying dynamical system.
There are several such a priori bounds in the literature, notably the
one due to Liverani \cite{liveraniannals}.
Although $|\lambda_2(\l)|$ is the optimal bound on the exponential
rate of mixing which holds for \emph{all} correlation functions with
holomorphic observables, faster exponential decay can occur for
observables in certain subspaces of finite codimension.  More
precisely, $|\lambda_n(\l)|$ bounds the exponential rate of mixing on
the subspace of observables with vanishing spectral projections
corresponding to $\lambda_2(\l), \ldots,\lambda_{n-1}(\l)$.  Therefore
the set of possible exponential rates of mixing (the \emph{correlation
  spectrum}, cf.~\cite{CPR}) is determined by the full eigenvalue
sequence $\{\lambda_n(\l)\}$.  Any a priori bounds on these
eigenvalues thus yields information on the finer mixing properties of
the underlying system.  The correlation spectrum is also closely
related to the \emph{resonances} of the underlying dynamical system
(see \cite{ruelleresonances1, ruelleresonances2}).

Explicit \emph{a priori} bounds on $\lambda_n(\L)$ also yield explicit
bounds on the Taylor coefficients of the determinant
$\det(I-\zeta\L)$, which in turn
facilitate a rigorous \emph{a posteriori} error analysis of any
computed approximations to the $\lambda_n(\L)$ (see \S
\ref{determinantsection} for details).
This rigorous justification of accurate numerical bounds has
applications to a number of topics in dynamical systems (e.g.~the
correlation spectrum \cite{CPR}, the linearised Feigenbaum
renormalisation operator \cite{aac, CCR, pollicottfeigenbaum},
Hausdorff dimension estimates \cite{jenkinsonpollicottjuliaklein}, the
Selberg zeta function for hyperbolic surfaces \cite{glz,
  mayermodular}, zeta functions for more general Anosov flows
\cite{fried}), as well as to other areas of mathematics
(e.g.~regularity estimates for refinable functions \cite{daubechies},
and the determinant of the Laplacian on surfaces of
negative curvature \cite{pollicottrocha}).

Our approach to explicitly bounding the eigenvalues of $\l$
is to consider completely general non-empty open subsets
$\Omega\subset\C^d$ in arbitrary complex dimension $d$, and
systematically work with Bergman space $A^2(\Omega)$, consisting of
those holomorphic
functions in $L^2(\Omega,dV)$, where $V$ denotes
$2d$-dimensional Lebesgue measure on $\Omega$.
For $\I$ a finite or countably infinite set, consider
a collection $(\phi_i)_{i\in\I}$ of holomorphic
maps\footnote{The $\phi_i$ here need not be
  complexified local inverses of some expanding map; in particular they
  need not be contractions with respect to the Euclidean metric.}
$\phi_i:\Omega\to\Omega$ such that
the closure of $\cup_{i\in\I}\phi_i(\Omega)$ is a compact subset of
$\Omega$,
and a collection $(w_i)_{i\in\I}$ of
functions $w_i\in A^2(\Omega)$ with
$\sum_{i\in \I}|w_i|\in L^2(\Omega,dV)$ (this condition obviously holds
whenever $\I$ is finite).  We then
call $(\Omega,\phi_i,w_i)_{i\in\I}$ a
{\it holomorphic map-weight system} on $\Omega$
and associate with it the transfer operator $\L$ defined as in
(\ref{introtransferopdefn}).
Our main result is:

\begin{procthm}[Theorem] If $\l:A^2(\Omega)\to A^2(\Omega)$ is
  the transfer operator corresponding to a holomorphic map-weight
  system $(\Omega,\phi_i,w_i)_{i\in\I}$ on a non-empty open set
  $\Omega\subset\C^d$, then
  \begin{equation}
\label{evaluestheorema}
  |\lambda_n(\l)|\le A\exp(-a n^{1/d})\quad\text{for all }n\in\N\,,
  \end{equation}
  where the constants $a,A>0$ can be determined explicitly in terms of
  computable properties of $(\Omega,\phi_i,w_i)_{i\in\I}$.
\end{procthm}
The above theorem is proved as Theorem~\ref{expliciteigenvaluebounds},
where
the coefficients $a,A>0$ are given explicitly. This theorem
is something of a folklore result.  Ruelle \cite[
p.~236]{ruelleinventiones} had originally asserted that the
eigenvalues of $\L$ tend to zero exponentially fast, following a claim
of Grothendieck \cite[II, Remarque 9, pp.~62--4]{grothendieck}.  In
1986 Fried \cite{fried} noted that in fact this assertion is false: in
dimension greater than one the eigenvalue decay rate can be slower
than exponential.  More precisely, for each dimension $d$ he exhibited
a transfer operator $\L$ whose eigenvalue sequence
satisfies (\ref{evaluestheorema}) for some $a,A>0$, but is not
$O(\exp(-an^\gamma))$ for any $\gamma>1/d$.
Recently, the bound (\ref{evaluestheorema}) has appeared
\cite[Thm.~4]{faureroy} in the setting of dynamical systems
on the torus, and also in
\cite[(3.6), p.~157]{glz})\footnote{The focus in \cite{glz}
is on the asymptotics
of the determinant with respect
to a complex parameter $s$,
rather than on
completely explicit
eigenvalue bounds.
In fact the derivation of the eigenvalue bound (3.6)
in \cite{glz} is not quite complete:
no argument is given for the bound
on the norm of the Bergman space operator
$\l^\rho_{ij}(s)$ \cite[p.~159]{glz}, and simple examples
(see \cite[\S 3.5]{cowenmaccluer}, \cite{koosmith})
show that in general the operator is not bounded.},
although in these papers the constants
$a$, $A$ are not given explicitly.
The bound (\ref{evaluestheorema}) is proved in
\cite{bandtlowjenkinsoncmp,bandtlowjenkinsonetds}, with explicit
formulae for $a$ and $A$, in the special case
where $\Omega$ is a Euclidean ball.

The
main contribution
of the present paper is a rigorous proof of
(\ref{evaluestheorema}) for arbitrary
non-empty open sets $\Omega\subset\C^d$,
including explicit upper bounds on the
positive constants $a$ and $A$.
The principal step towards proving (\ref{evaluestheorema}) consists of
establishing the estimate
\begin{equation}
\label{singintro}
s_n(\l)\le B\exp(-bn^{1/d})
\quad\text{for all }n\in\N\,,
\end{equation}
for explicit $b,B>0$, where $s_n(\l)$ denotes the $n$-th singular
value of $\l$.  The proof of (\ref{singintro}) consists of the
following three stages. In \S \ref{embeddingsection} the
analogous singular value estimate is first derived for canonical embedding
operators between Bergman spaces on \emph{strictly circled} domains.
In \S
\ref{disjointificationssection} the
result is established
for canonical identification
operators
$J$ between Bergman spaces on arbitrary non-empty open
subsets
$\Omega_2\subset \Omega_1\subset\C^d$, subject to the condition that
the closure of $\Omega_2$ is a compact subset of $\Omega_1$.
For this we introduce the notion of
a \emph{relative cover} of the pair $(\Omega_1,\Omega_2)$ by strictly
circled domains.
To each relative cover is associated its
\emph{efficiency}, a quantity which is readily computable,
and which can be used to explicitly bound the singular values of $J$.
In \S \ref{transferoperatorssection}, by factorising $\l$ as the
product of a bounded operator and a canonical identification operator,
we arrive at an explicit version of (\ref{singintro}).

Having established (\ref{singintro}), there are two possible routes to
deducing the eigenvalue bound (\ref{evaluestheorema}).  The first,
suggested by Grothendieck \cite{grothendieck}, and sketched in more
detail by Fried \cite[pp.~505--7]{fried}, is based on growth estimates
for the determinant $\det(I-\zeta\L)$.  We instead take a more direct approach
by applying Weyl's multiplicative inequality, relating eigenvalues
to singular values (see \S\ref{transferoperatorssection}).
For completeness we develop the
Grothendieck-Fried strategy as Appendix B.
Section~\ref{determinantsection} contains explicit bounds on the
Taylor coefficients of the determinant $\det(I-\zeta\l)$, derived from
the singular value estimates of
\S\ref{transferoperatorssection}, together with an outline of how
these bounds can be used to obtain explicit a \textit{a posteriori}
error bounds for spectral approximation procedures applied to transfer
operators.
Finally, in Appendix~A we show how our
Theorem~\ref{embeddingtheorem} can be used
to provide a short proof
of the correct
statement of Grothendieck's Remarque 9, which
does not seem to have
appeared in the literature yet:
if $L$ is any bounded linear
operator on the Fr\'echet space ${\mathcal  H}(\Omega)$ of holomorphic
functions on an open set $\Omega\subset\C^d$, then its eigenvalues are
$O(\exp(-an^{1/d}))$ as $n\to\infty$, for some $a>0$.

The methods of this paper
can be extended to prove an analogue of
the main result for more general
transfer operators arising
in the study of
limit sets of iterated function schemes
\cite{mauldinurbanski1,mauldinurbanski2} or of certain Kleinian groups
\cite{jenkinsonpollicottjuliaklein}, and whenever the underlying
dynamical system is a real analytic expanding Markov map. We do not pursue this
generalisation here, however, preferring to present the main ideas in
the simplest possible combinatorial setting.

\begin{notation}

Let $\N$ denote the set of strictly positive integers, and set
$\N_0:=\N\cup\{0\}$.
For $d\in\N$, let $\open$ denote the collection
of non-empty open subsets of $\C^d$.

For Hilbert spaces $H_1, H_2$,
let ${L}(H_1,H_2)$ denote the Banach space
of bounded linear operators from $H_1$ to $H_2$, equipped with the
usual norm, and let $\Si(H_1,H_2)\subset {L}(H_1,H_2)$ denote
the closed subspace of compact operators from $H_1$ to $H_2$. We
write $L$ or $\Si$ whenever the Hilbert spaces $H_1$
and $H_2$ are understood.

For $A\in\Si(H,H)$ let
$\lambda(A)=\set{\lam_n(A)}_{n=1}^\infty$ denote the sequence of
eigenvalues of $A$, each eigenvalue repeated according to its
algebraic multiplicity, and ordered by magnitude
(where distinct eigenvalues of the same modulus can be written in
any order), so that
$\abs{\lam_1(A)}\geq\abs{\lam_2(A)}\geq\ldots$.
For $A\in\Si(H_1,H_2)$, we define
the $n$-th
 \textit{singular value} of $A$ by
$s_n(A):= \sqrt{\lambda_n(A^*A)}$,
$n\ge1$, where $A^*$ denotes the adjoint of $A$.

\end{notation}

\section{Preliminaries}
\label{prelimiariessection}

\subsection{Exponential classes}
Much modern work on eigenvalue
distributions has been carried out within the framework of
\emph{operator ideals}
(cf.~\cite{gohbergkreinbook, pietsch1,pietsch2,simon}).  This
framework, however, is not well adapted to our setting:
as we shall
see, the transfer operators considered here are always trace
class (see Theorem~\ref{mainthm}),
and hence belong to any symmetrically normed ideal
(see e.g.~\cite[Chap.~3.2]{gohbergkreinbook}),
so that the results from this theory are too
conservative.
We instead use the theory of exponential classes
developed in \cite{expoclass}.

\begin{defn} Let $H_1, H_2$ be
infinite dimensional
Hilbert spaces.
For $a, \alpha > 0$, define
\[ E(a,\alpha):=\all{A\in\Si(H_1,H_2)}{ |A|_{a,\alpha}:=\sup_{n\in\N}
  s_n(A)\exp(an^{\alpha}) < \infty}\,, \]
the {\it  exponential class of operators of type}
$(a, \alpha)$.
Define
$E(\alpha):=\cup_{a >0}E(a,\alpha)$.
\end{defn}

Exponential classes enjoy
the following closure properties
(see \cite[Propositions 2.5 and 2.8]{expoclass}):

\begin{lem} \label{expoprop2}
Let  $\alpha, a, a_1,\ldots, a_N >0$.
  \begin{itemize}
  \item[(i)] If $A ,C\in L$ and $B\in E(a,\alpha)$, then $ABC
  \in E(a, \alpha)$, and
 $|ABC|_{a,\alpha}\leq \norm{A}{}\, |B|_{a,\alpha}\, \norm{C}{}$.
In particular,
$L E(a,\alpha) L\subset E(a,\alpha)$.
\item[(ii)] Let $A_n\in E(a_n,\alpha)$ for $1\leq n\leq
  N$ and let
  $A =\sum_{n=1}^N A_n$. Then
\[ A \in E(a',\alpha)\text{ with } |A|_{a',\alpha}\leq
  N \max_{1\leq n\leq N}|A_n|_{a_n,\alpha} \]
  where $a':=(\sum_{n=1}^Na_n^{-1/\alpha})^{-\alpha}$.
In particular,
$E(a_1,\alpha)+\cdots + E(a_N,\alpha)\subset E(a',\alpha)$,
and this inclusion is sharp in the sense that
$E(a_1,\alpha)+\cdots + E(a_N,\alpha)\not\subset E(b,\alpha)$
whenever $b>a'$.
\end{itemize}
\end{lem}

\subsection{Bergman spaces}
\label{bergmansection}
Bergman spaces, originally introduced by Stefan Bergman in his
1921 PhD
thesis \cite{bergmanthesis},
are among the simplest examples of Hilbert
spaces of holomorphic functions. Less delicate in their definition
than Hardy spaces, they provide a convenient setting for our analysis
of transfer operators.
\begin{defn}
  For $\Omega\in\open$, let $\H(\Omega)$ denote the
  Fr\'echet space of holomorphic
  functions $f:\Omega\to\C$, equipped with the topology of uniform
  convergence on compact subsets of $\Omega$.
  Let $A^\infty(\Omega)$ be the Banach space of bounded $f\in
  \HH(\Omega)$, equipped with the norm
  $\|f\|_{A^\infty(\Omega)}:=\sup_{z\in\Omega}|f(z)|$.
If $V$ denotes
  $2d$-dimensional Lebesgue measure on $\C^d$, normalised so that the
   $2d$-dimensional Euclidean unit ball has unit mass,
\[ A^2(\Omega):=\all{ f\in \H(\Omega)}%
{ \int_{\Omega}\abs{f(z)}^2\,dV(z) < \infty} \] is called {\it Bergman
  space over} $\Omega$.
\end{defn}

This definition of Bergman space is slightly more general then the
usual one, in that we allow arbitrary
non-empty open sets rather than just
domains. However most of their familiar properties (see e.g.
\cite[Chapter 1.4]{Kra}) are easily seen to
carry over to the
more general setting. In particular, $A^2(\Omega)$ is a separable
Hilbert space with inner product
\[ (f,g)_{A^2(\Omega)}=\int_\Omega f(z) \overline{g(z)} \, dV(z) \quad
(f,g\in A^2(\Omega)). \]

The following quantitative refinement of a well known lemma (see
\cite[Lemma 1.4.1]{Kra}) will be used in Lemma \ref{compopnorm}.

\begin{lem} \label{berg:lem1}
  If $\Omega\in\open$, and $K\subset\Omega$ is compact,
  there is a constant $C_K>0$ such that
$\sup_{z\in K}\abs{f(z)}\leq C_K \norm{f}{A^2(\Omega)}$
for all
$f\in A^2(\Omega)$.
Moreover, it is possible to choose
$ C_K=r^{-d}$,
where $r={\rm dist}(\partial K,\partial \Omega) ={\rm
  dist}(K,\partial\Omega)$.
\end{lem}

\begin{proof}
  By hypothesis, $r>0$ and $B(z,r)\subset \Omega$ for every $z\in K$,
  where $B(z,r)$ denotes the Euclidean ball of radius $r$ centred at
  $z$.
  If $f\in A^2(\Omega)$ then, just as in the standard
  case (see \cite[Lemma 1.4.1]{Kra}),
$f(z)=
(\int_{B(z,r)}f\,dV) / V(B(z,r))$,
so by the Cauchy-Schwarz inequality,
\[
  \abs{f(z)}\le\frac{1}{V(B(z,r))}\int_{B(z,r)}\abs{f}\,dV
  \leq V(B(z,r))^{-1/2}\norm{f}{L^2(B(z,r))}
  \leq r^{-d}\norm{f}{A^2(\Omega)}.
\]
\end{proof}

\section{Canonical embeddings for simple geometries}
\label{embeddingsection}

Suppose that $\Omega_1 , \Omega_2\in\open$, and
that $\Omega_2\subset \Omega_1$.  By restriction to $\Omega_2$ every
element in $A^2(\Omega_1)$ can also be considered as an element of
$A^2(\Omega_2)$. This restriction yields a linear transformation
$J:A^2(\Omega_1)\rightarrow A^2(\Omega_2)$ defined by
$Jf=f|_{\Omega_2}$, which will be referred to as {\it canonical
  identification} (we use $J$ throughout to denote canonical
  identifications; the spaces involved will always be clear from the
  context). If $\Omega_1$ is connected, then the canonical
identification is a proper embedding of $A^2(\Omega_1)$ in
$A^2(\Omega_2)$. Clearly $J$ is continuous,
with norm at most one.

\begin{defn}
  For $\Omega_1, \Omega_2\in\open$, if
  $\overline{\Omega}_2$ is a compact subset of $\Omega_1$ then we say
  that $\Omega_2$ is \emph{compactly contained} in $\Omega_1$, and
  write $\Omega_2 \cc \Omega_1$.
\end{defn}

It turns out that if $\Omega_2\cc\Omega_1$ then $J:A^2(\Omega_1)\to
A^2(\Omega_2)$ is a compact operator; to see this note that
$J(A^2(\Omega_1))$ is contained in the
Banach space $C^b(\Omega_2)$ of bounded continuous functions on
$\Omega_2$ and $J:A^2(\Omega_1) \to C^b(\Omega_2)$ has
closed graph, hence $\{Jf: \|f\|_{A^2(\Omega_1)} \leq 1\}$
is uniformly bounded
on $\Omega_2$ and therefore a normal family in $A^2(\Omega_2)$.
In fact rather more is true: $J\in E(c,1/d)$ for some $c>0$.
The proof of this result
for general open sets $\Omega_2 \cc\Omega_1$ requires a certain amount
of preparation and will be presented in
\S \ref{disjointificationssection}.
In this section we shall be content with proving the result
for certain subclasses of open sets $\Omega_1, \Omega_2$ for which the
decay rate $c$ can be identified precisely; these subclasses are
defined as follows.

\begin{defn}
  Let $D\subset\C^d$ and $\zeta\in\C^d$.  We call $D$ {\it strictly circled,
    with centre $\zeta$}, if
\[ \mu (D-\zeta)\cc D-\zeta \quad \text{ for all $\mu\in \C$ with
  $\abs{\mu}< 1$}. \]
For $r>0$ we define
$D(r):=r(D-\zeta)+\zeta$.
\end{defn}
Note that a strictly circled set is necessarily bounded. Moreover,
the boundary of a strictly
circled open set has zero Lebesgue measure, a fact
which will be used
in \S\ref{disjointificationssection}.

\begin{lem}
\label{StrictlyCircledZeroLeb}
If $D\in \open$ is strictly circled then $V(\partial D)=0$.
\end{lem}
\begin{proof}
  By translation invariance of $V$,
it suffices to prove the assertion for $D$ with centre 0.
Since $D$ is open,
$D=\cup_{0<r<1}\, r\overline{D}$,
where
$\overline{D}$
denotes the closure of $D$.
Thus
$
V(D)=\sup_{0<r<1}V(r\overline{D})=
\sup_{0<r<1}r^{2d}V(\overline{D})=V(\overline{D})$.
\end{proof}
We now consider canonical embeddings of Bergman
spaces over strictly circled open sets.
\begin{prop}\label{berg:prop5}
  If $D\in\open$ is strictly circled then:
\begin{itemize}
\item[(i)] There is a set consisting
  of homogeneous polynomials which is a complete orthogonal system for
  every $A^2(D(r))$, $r>0$.
\item[(ii)] If $\gamma>1$, then the
  singular values of the canonical embedding
$J:A^2(D(\gamma))\hookrightarrow A^2(D)$
are given by
$s_n(J)=\gamma^{-(k+d)}$ for
  ${k+d-1\choose d}< n \leq {k+d\choose d}$ and $k\in\N_0$.
\end{itemize}
\end{prop}

\begin{proof}
  (i) Assume for the moment that $D$ is centred at the origin.
  Any function holomorphic on the strictly circled set
$D(r)$ has a
  unique expansion in terms of homogeneous polynomials, which is
  convergent uniformly on compact subsets of $D(r)$ (see
  \cite[Chapter~I, \S10.3, Thm.~2]{andreiancazacu} or \cite[Chapter
  II, Thm.~3]{malgrange}), hence the collection of homogeneous
  polynomials is total (i.e.~its linear span is dense) in $A^2(D(r))$.
  It remains to show that this collection can be orthogonalised so as
  to yield a system that is simultaneously orthogonal for all
  $A^2(D(r))$, $r>0$.
  To do this we introduce the short-hand
$(f,g)_r:=(f,g)_{A^2(D(r))}$.
Let $f$ and $g$ be monomials of degree $n$ and $m$ respectively. Since
$D$ is bounded, $f,g\in A^2(D(r))$ for all $r>0$. Moreover, since $D$ is
strictly circled, each $D(r)$ is invariant under the transformation $z\mapsto
e^{it}z$.  Thus
\[
(f,g)_r=\int_{D(r)}f(e^{it}z)
\overline{g(e^{it}z)}\,dV(z)=e^{it(n-m)}(f,g)_r\,, \] which implies
$(f,g)_r=0$ for $n\neq m$, and for each $r>0$.

For any $r_1>0$, an application of the Gram-Schmidt orthogonalisation
procedure now yields an orthonormal basis of $A^2(D(r_1))$ consisting
of homogeneous polynomials. We shall show that this basis is also
orthogonal with respect to all other scalar products
$(\cdot,\cdot)_{r}$, for $r>0$. To see this fix $r>0$ and let $f$ and
$g$ be homogeneous polynomials of degree $n$ and $m$ respectively.
Then
\begin{align}\label{berg:propx:e1}
  (f,g)_{r}
  &=\int_{D(r)}f(z)\overline{g(z)} \,dV(z) \nonumber \\
  &= \int_{D(r_1)}f((r/r_1)z)\overline{g((r/r_1)z)}(r/r_1)^{2d}\,dV(z)
  \\
  &=(r/r_1)^{n+m+2d}(f,g)_{r_1}\,.  \nonumber
   \end{align} Thus, if $(f,g)_{r_1}=0$
   then $(f,g)_r=0$, and (i) is proved.

   \smallskip (ii) If $J^*:A^2(D)\rightarrow A^2(D(\gamma))$ denotes
   the adjoint of $J:A^2(D(\gamma))\hookrightarrow A^2(D)$, then
   setting $r=1$, $r_1=\gamma$ in (\ref{berg:propx:e1}) gives
   $
   (f,J^*Jg)_{\gamma} =(Jf,Jg)_{1}
   =\gamma^{-(n+m+2d)}(f,g)_{\gamma}
   $.
   Thus $J^*J:A^2(D(\gamma))\rightarrow A^2(D(\gamma))$ is diagonal
   with respect to the orthogonal basis of homogeneous polynomials.
   Its eigenvalues therefore belong to the set
   $\all{\gamma^{-(2k+2d)}}{k\in\N_0}$.  Therefore the singular values
   of $J:A^2(D(\gamma))\hookrightarrow A^2(D)$ belong to the set
   $\all{\gamma^{-(k+d)}}{k\in\N_0}$. As there are $k+d-1\choose d-1$
   linearly independent homogeneous polynomials of degree $k$, the
   value $\gamma^{-(k+d)}$ occurs with multiplicity $k+d-1\choose
   d-1$. Thus the largest $n$ for which $s_n(J)=\gamma^{-(k+d)}$ is
   equal to $\sum_{l=0}^k{l+d-1\choose d-1}={k+d \choose d}$.  This
   completes the proof in the case of $D$ centred at $0$. The general
   case can be reduced to this case by shifting the origin and using
   translation invariance of Lebesgue measure.
\end{proof}

The precise location
   of $J$ in the scale of exponential classes $\set{E(a,\alpha)}$
is as follows:

\begin{prop} \label{berg:prop:sasymp}
  If $D\in\open$ is strictly circled, and
  $\gamma>1$, then the canonical embedding
  $J:A^2(D(\gamma))\hookrightarrow A^2(D)$ satisfies
\begin{equation}
\label{berg:prop6:e3}
J\in E(c,1/d),\quad\text{where }
c=(d!)^{1/d} \log \gamma \,,
\end{equation}
and
\begin{equation}
|J|_{c,1/d}=\gamma^{(1-d)/2} \label{berg:prop6:e3i}
\end{equation}

That is, its singular value sequence has the following
asymptotics:
\begin{equation}
\lim_{n\rightarrow\infty}\frac{\log\abs{\log s_n(J)}}{\log
n}=\frac{1}{d}\,; \label{berg:prop6:e1}
\end{equation}

\begin{equation}
\lim_{n\rightarrow\infty} \frac{ \log
s_n(J)}{n^{1/d}}= - (d!)^{1/d}\log \gamma\, ; \label{berg:prop6:e2}
\end{equation}

\begin{equation}
\label{berg:prop6:e4}
\sup_{n\in\N} \left ( \log s_n(J) + (nd!)^{1/d}\log \gamma \right
)=\frac{1-d}{2}\log \gamma\,.
\end{equation}
\end{prop}
\begin{proof}
  If $h_d(k):={k+d-1\choose d}$,
and $h_d(k)<n\le h_d(k+1)$,
Proposition
  \ref{berg:prop5} gives
\begin{equation*}
\frac{ \log \abs{\log
\gamma^{-1}} + \log (k+d)}{\log h_d(k+1)}\leq
\frac{ \log \abs{\log s_n(J)}}{\log n}\leq
\frac{ \log \abs{\log
\gamma^{-1}} + \log (k+d)}{\log h_d(k)}\,.
\end{equation*}
It is easily seen that
$
\lim_{k\rightarrow \infty}\frac{\log (k+d)}{\log h_d(k+1)}
=\lim_{k\to\infty}\frac{\log
(k+d)}{\log h_d(k)}=\frac{1}{d}$,
so (\ref{berg:prop6:e1}) follows.
Similarly,
\begin{equation*}
\frac{(k+d)\log \gamma^{-1}}{h_d(k+1)^{1/d}}
\leq \frac{\log s_n(J)}{n^{1/d}}
\leq \frac{(k+d)\log \gamma^{-1}}{h_d(k)^{1/d}}\,,
\end{equation*}
and
$
\lim_{k\rightarrow
\infty}\frac{(k+d)}{h_d(k+1)^{1/d}}=\lim_{k\to\infty}\frac{
(k+d)}{h_d(k)^{1/d}}=(d!)^{1/d}
$,
so (\ref{berg:prop6:e2}) follows.

To prove (\ref{berg:prop6:e4}), we first establish that for all
$d\in\N$,
\begin{equation}
\label{berg:techlem}
\sup_{x\geq 0}
\prod_{j=1}^d(x+j)^{1/d}-(x+d)=
\lim_{j\rightarrow\infty}\prod_{j=1}^d(x+j)^{1/d}-(x+d)=-\frac{d-1}{2}\,.
\end{equation}
The case $d=1$ of (\ref{berg:techlem}) is obvious, so suppose
$d\geq 2$.  If $h(x):=\prod_{j=1}^d(x+j)^{1/d}-(x+d)$
then
\[
h'(x)=\frac{1}{d}\prod_{j=1}^d(x+j)^{1/d-1}\sum_{j=1}^d\prod_{\substack{l=1
    \\ l\neq j}}(x+l)-1. \] Now
$\frac{1}{d}\sum_{j=1}^d(x+j)^{-1}\geq \prod_{j=1}^d(x+j)^{-1/d}$
by the arithmetic-geometric mean inequality, so
\[ \frac{1}{d}\sum_{j=1}^d\prod_{\substack{l=1 \\ l\neq j}}(x+l)
\geq \prod_{j=1}^d(x+j)^{1-1/d}\,, \]
and therefore $h'(x)\geq 0$ for
$x\geq 0$.
If $t:=(x+d)^{-1}$ then
\[ h(x)=(x+d)\left ( \prod_{j=1}^d \left ( \frac{x+j}{x+d} \right
  )^{1/d}-1 \right )= t^{-1}\left ( \prod_{j=0}^{d-1}(1-jt)^{1/d}-1
\right )\,, \]
so
$\sup_{x\ge0}h(x)=
\lim_{x\rightarrow \infty}h(x)=
\lim_{t \rightarrow 0}t^{-1}\left ( \prod_{j=0}^{d-1}(1-jt)^{1/d}-1
\right )= -\frac{1}{d}\sum_{j=0}^{d-1}j=-\frac{d-1}{2}$ by
l'H\^{o}pital's rule, and (\ref{berg:techlem}) is proved.

Now
$
  \log s_n(J)+(nd!)^{1/d}\log \gamma
  \leq  \left ( (h_d(k+1)d!)^{1/d}-(k+d) \right )\log \gamma
   \leq \left ( -\frac{d-1}{2} \right )\log \gamma
$,
by (\ref{berg:techlem}),
so $\sup_{n\in\N} \left ( \log
  s_n(J) + (nd!)^{1/d}\log \gamma \right )\leq \frac{1-d}{2}\log
\gamma$. To obtain equality we consider $s_{h_d(k+1)}(J)$ and again
apply (\ref{berg:techlem}).
Finally, note that (\ref{berg:prop6:e4}) is a restatement of
(\ref{berg:prop6:e3}) and (\ref{berg:prop6:e3i}).
\end{proof}

\begin{rem}
Proposition \ref{berg:prop:sasymp} is optimal:
as a consequence of (\ref{berg:prop6:e1}) and
(\ref{berg:prop6:e2}), membership of $J$ in (\ref{berg:prop6:e3}) is
sharp, in the sense that neither $1/d$ nor $c$ can be replaced by
anything larger; moreover, $|J|_{c,1/d}$ is known exactly.
\end{rem}

\section{Canonical identifications and relative covers}
\label{disjointificationssection}

We shall now show how identifications of Bergman spaces over more
general sets can be obtained from identifications of Bergman spaces
over strictly circled sets. The main tool is the
following construction:

\begin{lem}\label{berg:lem3}
  If $\Omega_1,\Omega_2,\Omega_3\in\open$, with
$\Omega_1\subset \Omega_2 \subset \Omega_3$, the
  operator
$T_{\Omega_1}: A^2(\Omega_2) \rightarrow A^2(\Omega_3)$,
defined implicitly by
$(T_{\Omega_1}f,g)_{A^2(\Omega_3)}=\int_{\Omega_1}f\,\overline{g}\,dV $,
is bounded with norm at most 1.
\end{lem}
\begin{proof}
  Notice that for any $f\in A^2(\Omega_2)$ and any $g\in
  A^2(\Omega_3)$
\[ \abs{\int_{\Omega_1} f\overline{g}\,dV}^2\leq \int_{\Omega_1}
\abs{f}^2\,dV \, \int_{\Omega_1}\abs{g}^2\, dV \leq \norm{f}{A^2(\Omega_2)}^2\,
\norm{g}{A^2(\Omega_3)}^2\,. \] Thus $T_{\Omega_1}$
is well-defined and
continuous with norm at most $1$.
\end{proof}

\begin{defn}
  Let $\{\Omega_n\}_{1\leq n\leq N}$ be a finite collection of open
  subsets of $\C^d$.  If $\{\widetilde\Omega_n\}_{1\leq n\leq N}$ is a
  partition (modulo sets of zero Lebesgue measure) of $\bigcup
  _{n=1}^N \Omega_n$, where each $\widetilde \Omega_n$ is open, and
  $\widetilde\Omega_n\subset \Omega_n$ for each $n$, then we say that
  $\{\widetilde\Omega_n\}_{1\leq n\leq N}$ is a
  \emph{disjointification} of $\{\Omega_n\}_{1\leq n\leq N}$.

\end{defn}

\begin{rem}
\label{DisjointificationExists}
If a collection $\set{\Omega_n}_{1\leq n\leq N}$ has the
  property that the boundary of each $\Omega_n$ is a Lebesgue null
  set, then a disjointification exists and can, for example, be
  obtained by defining
  $
  \widetilde\Omega_n$
  as the interior of $\left( \Omega_n\setminus
    (\cup_{i=1}^{n-1} \Omega_i)\right)$ for $n=1,\ldots,N$.
\end{rem}
The usefulness of the operator defined in Lemma~\ref{berg:lem3} is
due to the following key
result.

\begin{prop}\label{berg:prop:top}
  For $1\le n\le N$, let $\Omega_n,\Omega\in\open$,
  with $\Omega_n\subset \Omega$, and let
$J_n: A^2(\Omega) \rightarrow A^2(\Omega_n)$
denote the canonical identification.
If $\{\widetilde\Omega_n\}_{1\leq n\leq N}$ is a disjointification of
$\{\Omega_n\}_{1\leq n\leq N}$, then the canonical identification
$J:A^2(\Omega)\rightarrow A^2(\bigcup_{n=1}^N\Omega_n)$
can be written as
\[ J=\sum_{n=1}^NT_{\widetilde\Omega_n}J_n\,, \]
where
$T_{\widetilde\Omega_n}:A^2(\Omega_n)\rightarrow
A^2(\bigcup_{n=1}^N\Omega_n) $ is the operator defined in Lemma
\ref{berg:lem3}.
\end{prop}
\begin{proof}
  Let $f\in A^2(\Omega)$ and $g\in A^2(\bigcup_{n=1}^N\Omega_n)$. Then
\begin{align*}
  (\sum_{n=1}^NT_{\widetilde\Omega_n}J_nf,g)_{A^2(\bigcup_n\Omega_n)}
  &=\sum_{n=1}^N\int_{\widetilde\Omega_n}f(z)\overline{g(z)}\,dV(z)\\
  &=\int_{\bigcup_n\Omega_n}f(z)\overline{g(z)}\,dV(z)\\
  &=(Jf,g)_{A^2(\bigcup_n\Omega_n)},
\end{align*}
and the assertion follows.
\end{proof}

\begin{defn}
\label{relativecoverdefn}
Let $\Omega_1,\Omega_2\in \open$, with $\Omega_2\cc
\Omega_1$, and $N\in\N$.  A finite collection $D_1,\ldots,D_N$ of
strictly circled open subsets of $\C^d$ is a {\it relative
  cover}
of the pair
$(\Omega_1,\Omega_2)$ if

  \begin{itemize}
  \item[(a)] $\Omega_2\subset \bigcup_{n=1}^ND_n\,$, and
  \item[(b)] for each $1\le n\le N$ there exists
    $\gamma_n>1$ such that $\bigcup_{n=1}^N
    D_n(\gamma_n)\subset \Omega_1$.
 \end{itemize}
 We call $N$ the {\it size}, $(\gamma_1,\ldots,\gamma_N)$
 a {\it scaling},
 and
$ \Gamma=(\log \gamma_1,\ldots,\log \gamma_N)$
the {\it efficiency}, of the relative cover.

\end{defn}

\begin{rem} Since $\Omega_2 \cc \Omega_1$,
there always exists a relative cover for $(\Omega_1,\Omega_2)$.
\end{rem}

\begin{theorem}
\label{embeddingtheorem}
If $\Omega_1,\Omega_2\in\open$, with $\Omega_2\cc
\Omega_1$,
then the canonical identification
$ J:A^2(\Omega_1)\rightarrow A^2(\Omega_2) $
belongs to $E(1/d)$.

\noindent
More precisely,
if $\{D_n\}_{1\le n\le N}$ is a relative
cover of $(\Omega_1,\Omega_2)$ of size $N$ with efficiency $\Gamma$
then
\[
J\in E(c,1/d), \quad \text{where}\quad c=\norm{\Gamma}{d},
\]
and
\begin{equation}
\label{jcguage}
|J|_{c,1/d}\leq N e^{-\frac{d-1}{2}\min(\Gamma)}\,,
\end{equation}
where $\min(\Gamma)$ denotes the smallest entry in $\Gamma$ and
$\norm{(x_1,\ldots,x_N)}{d}:=\left(\sum_{j=1}^N |x_j|^{-d}\right )^{-1/d}$.
\end{theorem}

\begin{proof}
For
$1\leq n\leq N$, let
$ T_{\widetilde\Omega_n}:A^2(D_n)\rightarrow A^2(\bigcup_{n=1}^ND_n)$
denote the operator defined in Lemma
\ref{berg:lem3},
where $\{\widetilde\Omega_n\}_{1\leq n\leq N}$
is a disjointification
of
  $\{D_n\}_{1\le n\le N}$
(which exists by Lemma~\ref{StrictlyCircledZeroLeb} and
  Remark~\ref{DisjointificationExists}).
For $(\gamma_1,\ldots,\gamma_N)$ a
  scaling of $\{D_n\}_{1\le n\le N}$, consider the
  canonical
identifications
$ \widetilde J_n:A^2(\Omega_1)\rightarrow A^2(D_n(\gamma_n))$,
$ J_n:A^2(D_n(\gamma_n))
\rightarrow A^2(D_n)$,
and
$\widetilde J:A^2
(\bigcup_{n=1}^ND_n)\rightarrow A^2(\Omega_2)$.
 By Proposition
\ref{berg:prop:top},
\begin{equation*}
 J=\sum_{n=1}^N\widetilde J T_{\widetilde\Omega_n}J_n\widetilde
J_n\,.
\end{equation*}

Trivially $\|\widetilde J\|\leq 1$ and
$\|\widetilde J_n\|\leq 1$, while
$\|T_{\widetilde\Omega_n}\|\leq 1$ by Lemma~\ref{berg:lem3}, so
Lemma~\ref{expoprop2} (i) and
Proposition \ref{berg:prop:sasymp} imply that each
$
\widetilde J T_{\widetilde\Omega_n}J_n\widetilde J_n \in E(c_n,1/d)$,
where $c_n=(d!)^{1/d}\log \gamma_n$,
and
$
|\widetilde J T_{\widetilde\Omega_n}J_n\widetilde J_n|_{c_n,1/d}\leq
\gamma_n^{-\frac{d-1}{2}}$.
 The assertion now follows from Lemma~\ref{expoprop2} (ii).
\end{proof}

\section{Singular values and eigenvalues of transfer operators}
\label{transferoperatorssection}

Our previous analysis
of the singular values of identification operators
can now be applied  to the study
of the singular values of transfer
operators.
Since a transfer
operator
can be expressed in terms of
multiplication operators and composition operators,
we begin by considering such operators.

\begin{defn}
  Let $\Omega,\widetilde{\Omega}\in\open$.
\begin{itemize}
\item[(a)] If $\phi:\Omega\rightarrow \widetilde{\Omega}$ is
holomorphic, the linear transformation
$C_\phi:{\mathcal  H}(\widetilde{\Omega})\to{\mathcal  H}(\Omega)$
defined by $C_\phi f:=f\circ \phi$
is called
a {\it composition operator} (with {\it symbol} $\phi$).
\item[(b)] If $w\in {\mathcal  H}(\Omega)$,
the linear transformation
$M_w:{\mathcal  H}(\Omega)\to {\mathcal  H}(\Omega)$
defined by
$(M_w f)(z):=w(z)f(z)$
is called a {\it multiplication operator} (with {\it symbol} $w$).
\item[(c)] An operator of the form $M_wC_\phi$, where $C_\phi$ is a
  composition operator and $M_w$ is a multiplication operator, is
  called a {\it weighted composition operator}.
\end{itemize}
\end{defn}

\begin{notation}
If $F,G$ are Banach spaces, and $A:F\to G$
is a bounded linear operator, the norm of $A$
will sometimes be denoted by $\norm{A}{F\to G}$.
\end{notation}

\begin{lem}
\label{compopnorm}
If $\Omega, \widetilde{\Omega}\in\open$, $\phi:\Omega
\rightarrow \widetilde{\Omega}$ is holomorphic, and
$r:={\rm dist}(\phi(\Omega),\partial \widetilde{\Omega})>0$,
then
$C_\phi:A^2(\widetilde{\Omega})\rightarrow A^\infty(\Omega) $
is bounded, with norm
\begin{equation}
\label{compopnorm:twoinfty}
\norm{C_\phi}{A^2(\widetilde{\Omega})\rightarrow A^\infty(\Omega)}\,
\le r^{-d} \,.
\end{equation}

If, in addition, $\Omega$ has finite volume, then
$C_\phi:A^2(\widetilde{\Omega})\rightarrow A^2(\Omega)$
is bounded, with norm
$$
\norm{C_\phi}{A^2(\widetilde{\Omega})\rightarrow A^2(\Omega)}
\le \sqrt{V(\Omega)}\, r^{-d} \,.
$$

\end{lem}
\begin{proof}
By Lemma \ref{berg:lem1},
$
  \|C_\phi f\|_{A^\infty(\Omega)}=\sup_{z\in\Omega}|f(\phi
  (z))|
 = \sup_{z\in \phi(\Omega)}|f(z)|
 \le r^{-d}\norm{f}{A^2(\widetilde{\Omega})}
$ for $f\in A^2(\widetilde{\Omega})$,
thus
$C_\phi$ maps
$A^2(\widetilde{\Omega})$ continuously to $A^\infty(\Omega)$,
with norm as in (\ref{compopnorm:twoinfty}).
The remaining assertions follow from the fact that if $\Omega$ has
finite volume then the canonical identification
$J:A^\infty(\Omega)\to A^2(\Omega)$
is continuous
with norm $\norm{J}{}=\sqrt{V(\Omega)}$.

\end{proof}

\begin{rem}
\item[\, (i)]
There is a sizable literature on criteria for continuity of
composition operators between Bergman spaces, beginning with
Littlewood's
subordination principle \cite{littlewood}, guaranteeing
that if $\Omega=\widetilde{\Omega}$ is the open unit disc
then $C_\phi$ is \emph{always} bounded
(see \cite[Prop.~3.4]{maccluershapiro}).
This need not be the case for more general simply connected
domains in $\C$
(see
\cite{kumarpartington, shapirosmith}), or indeed
when
$\Omega=\widetilde{\Omega}$ is the open unit ball
in $\C^d$, $d>1$
(see e.g.~\cite[\S 3.5]{cowenmaccluer}).
A novelty of our approach is that we consider Bergman
spaces in arbitrary dimension, and over arbitrary open sets $\Omega$.
\item[\, (ii)]
  There is no known general formula (in terms of the
symbol $\phi$) for the norm of the composition operator $C_\phi$;
see \cite[p.~195]{shapiro} for a discussion of this problem.
\end{rem}

Next we consider weighted composition operators. Again we may ask
under what conditions $M_wC_\phi$ maps $A^2(\widetilde{\Omega})$
continuously into
$A^2(\Omega)$.
A necessary condition
is that
$w\in A^2(\Omega)$, since the image of the
constant function 1 is $w$.
In general this is not enough to guarantee the
continuity of $M_w$ itself (in one complex
  dimension, necessary and sufficient conditions for the continuity of
  multiplication operators are
  given in \cite{kumarpartington}), but in our context
it \emph{is}
sufficient for the boundedness of $M_wC_\phi$:

\begin{lem}
\label{weightedcompopnorm}
Suppose $\Omega,\widetilde{\Omega}\in\open$, $\phi:\Omega\to
\widetilde{\Omega}$ is holomorphic, and
$r:={\rm dist}\,(\phi(\Omega),\partial \widetilde{\Omega})>0$.
If $w\in
A^2(\Omega)$, the weighted composition operator
$M_wC_\phi:A^2(\widetilde{\Omega})\to A^2(\Omega)$
is bounded, with
\[ \norm{M_wC_\phi}{A^2(\widetilde{\Omega})\to A^2(\Omega)}\leq
r^{-d}\norm{w}{A^2(\Omega)}\,.\]
\end{lem}
\begin{proof}
  If $f\in A^2(\widetilde{\Omega})$ then $w\cdot(f\circ \phi)\in {\mathcal
  H}(\Omega)$.
Now
$\sup_{z\in\Omega}\abs{f(\phi(z))}^2\leq
  r^{-2d}\norm{f}{A^2(\widetilde{\Omega})}^2$
by Lemma~\ref{compopnorm}, so
$
\|M_wC_\phi f\|_{A^2(\Omega)}^2
=\int_\Omega\abs{w(z)}^2
\abs{f(\phi(z))}^2\,dV(z)
 \leq
 r^{-2d}\|f\|_{A^2(\widetilde{\Omega})}^2\|w\|_{A^2(\Omega)}^2
$.
\end{proof}

\begin{defn}
  Let $\Omega, \widetilde{\Omega}\in\open$,
and let $\I$ be either a finite or
  countably infinite set. Suppose we are given the following data:
\begin{itemize}
\item[(a)] a collection $(\phi_i)_{i\in I}$ of holomorphic maps $\phi_i:\Omega
    \to\widetilde{\Omega}$
with
$\cup_{i\in \I}\phi_i(\Omega)\cc\widetilde{\Omega}$;
\item[(b)] a collection $(w_i)_{i\in\I}$ of functions $w_i\in
    A^2(\Omega)$ with
$ \sum_{i\in \I}|w_i|\in L^2(\Omega,dV)$,
i.e.~the series of the moduli of the $w_i$ converges in
$L^2(\Omega,dV)$.
\end{itemize}
We then call $((\Omega,\widetilde{\Omega}),\phi_i,w_i)_{i\in\I}$ a
{\it holomorphic map-weight system (on
  $(\Omega,\widetilde{\Omega})$)}.  If $\widetilde{\Omega}=\Omega$
then we simply refer to a {\it holomorphic map-weight system on
  $\Omega$}, denoted by $(\Omega,\phi_i,w_i)_{i\in\I}$.
\end{defn}

To each holomorphic map-weight system we associate a transfer operator
as follows (note that when $\widetilde{\Omega}=\Omega$, the definition
coincides with the one given in \S \ref{introsection}):
\begin{defn}
\label{holotransfopdefn}
Let $((\Omega,\widetilde{\Omega}),\phi_i,w_i)_{i\in\I}$ be a
holomorphic map-weight system.  Then the linear operator
$\trop:A^2(\widetilde{\Omega})\to A^2(\Omega)$ defined as the sum of
weighted composition operators
\[ \trop =\sum_{i\in\mathcal  I}M_{w_i}C_{\phi_i}\,, \]
is called the associated
{\it transfer operator (on $(\Omega,\widetilde{\Omega})$)}.
\end{defn}

If $\I$ is infinite, it is not obvious that this definition of $\l$
produces a well-defined continuous operator from
$A^2(\widetilde{\Omega})$ to $A^2(\Omega)$.  We now prove that
this is indeed the case.

\begin{prop}
\label{sumweightedcompop}
Let $((\Omega,\widetilde{\Omega}),\phi_i,w_i)_{i\in\I}$ be a
holomorphic map-weight system,
with
$r_i:={\rm dist}\,(\phi_i(\Omega),\partial \widetilde{\Omega})$.
 The associated transfer operator
$\l:A^2(\widetilde{\Omega})\to A^2(\Omega)$
is bounded, with norm
\begin{equation}\label{withnorm}
 \norm{\l}{A^2(\widetilde{\Omega})\to A^2(\Omega)}\leq
\norm{\sum_{i\in\I}\abs{w_i}r_i^{-d}}{L^2(\Omega)}\,.
\end{equation}
 \end{prop}
\begin{proof}
  Let $f\in A^2(\widetilde{\Omega})$. If $\J\subset\I$
   is finite,
  $\sum_{i\in\J}M_{w_i}C_{\phi_i}f\in A^2(\Omega)$ by
  Lemma~\ref{weightedcompopnorm}.
Now
\begin{equation*}
  \norm{\sum_{i\in\J}M_{w_i}C_{\phi_i}f}{A^2(\Omega)}^2
   \leq \int_\Omega \left
    (\sum_{i\in\J}|w_i(z)|\,|f(\phi_i(z))|\right )
  ^2\,dV(z)\,,
\end{equation*}
and
$\sup_{z\in \Omega}\abs{f(\phi_i(z))}\leq
r_i^{-d}\|f\|_{A^2(\widetilde{\Omega})}$
by Lemma~\ref{compopnorm}, so
\begin{equation}
\norm{\sum_{i\in\J}M_{w_i}C_{\phi_i}f}{A^2(\Omega)}^2
\leq  \norm{f}{A^2(\widetilde{\Omega})}^2\int_\Omega \left
(\sum_{i\in\J}|w_i(z)|r_i^{-d}\right )^2\,dV(z)\,.
\label{weightedcompop:mainineq}
\end{equation}
Since each $r_i\geq {\rm dist}\,( \cup_{i\in\I}\phi_i(\Omega),\partial
\widetilde{\Omega})=:r>0$,
\[ \int_\Omega \left (\sum_{i\in\J}|w_i(z)|r_i^{-d}\right )^2\,dV(z)
\leq r^{-2d} \int_\Omega \left (\sum_{i\in\J}|w_i(z)|\right
)^2\,dV(z)\,.
\]
So (\ref{weightedcompop:mainineq}) implies that
$\sum_{i\in\I}M_{w_i}C_{\phi_i}f$ is Cauchy in $A^2(\Omega)$,
hence converges to an element in $A^2(\Omega)$. Thus $\l$ defines a
bounded operator from $A^2(\widetilde{\Omega})$ to $A^2(\Omega)$, by
the uniform boundedness principle. Choosing $\J=\I$ in
(\ref{weightedcompop:mainineq}) yields the desired upper bound
on the norm of $\l$.
\end{proof}

We now prove
that for any holomorphic
map-weight system, the corresponding transfer operator
lies in an exponential class $E(c,1/d)$, with
explicit estimates on both $c$ and
$|\l|_{c,1/d}$:

\begin{theorem}
\label{mainthm}
Suppose that $((\Omega,\Omega'),\phi_i,w_i)_{i\in\I}$ is a
holomorphic map-weight system with $\Omega,\Omega'\in\open$,
and $r_i:={\rm dist}(\phi_i(\Omega),\partial\widetilde{\Omega})$.
  Let
$\widetilde{\Omega}\cc\Omega'$ be such that
\begin{equation}
\label{twoccs}
\cup_{i\in\I}\phi_i(\Omega) \cc \widetilde{\Omega}\cc\Omega'\,,
\end{equation}
and such that $(\Omega',\widetilde{\Omega})$ has a relative cover of
size $N$ with efficiency $\Gamma$.

Then the corresponding transfer operator $\l:A^2(\Omega')\to A^2(\Omega)$
belongs to the exponential class $E(c,1/d)$, where
\begin{equation}
\label{cformula}
c =\norm{\Gamma}{d},
\end{equation}
and
\begin{equation}\label{trans:eq36b}
  |\l|_{c,1/d} \le N
  e^{-\frac{d-1}{2}\min(\Gamma)}
\norm{\sum_{i\in\I}|w_i|r_i^{-d}}{L^2(\Omega)}\,.
\end{equation}
\end{theorem}

\begin{proof}
By Proposition~\ref{sumweightedcompop} the transfer operator
  $\l:A^2(\Omega')\to A^2(\Omega)$ can be lifted to a continuous
  operator
$ \widetilde\l:A^2(\widetilde{\Omega})\to A^2(\Omega)$.
If $J:A^2(\Omega')\rightarrow A^2(\widetilde{\Omega})$ denotes the
canonical identification, $\l$ factorises as $\l =
\widetilde\l J $.
By Theorem \ref{embeddingtheorem},
$J\in E(c,1/d)$, where $c$ is as in (\ref{cformula}), and
(\ref{jcguage}) gives
$|J|_{c,1/d}\leq Ne^{\frac{1-d}{2}\min(\Gamma)}$.
Lemma~\ref{expoprop2} now shows that $\l=\widetilde\l J\in
E(c,1/d)$, with
$
|\l|_{c,1/d} \le
\|\widetilde\l\|_{A^2(\widetilde{\Omega})\to
  A^2(\Omega)}\, |J|_{c,1/d}
$,
and (\ref{withnorm})
yields the desired bound for
$|\l|_{c,1/d}$.
\end{proof}

\begin{rem}
  In Theorem \ref{mainthm} there is some freedom in the choice of
  $\widetilde{\Omega}$.
  The condition $\cup_{i\in\I}\phi_i(\Omega)\cc\widetilde{\Omega}$
  ensures that $\widetilde\l:A^2(\widetilde{\Omega})\to
  A^2(\Omega)$ is bounded, while
  $\widetilde{\Omega}\cc\Omega'$ is required so that
   $J:A^2(\Omega')\rightarrow A^2(\widetilde{\Omega})$
  lies in some exponential class $E(c,1/d)$.  In practice the choice
  of $\widetilde{\Omega}$ subject to (\ref{twoccs}) would be made
  according to the relative importance of a sharp bound on $c$ or on
  $|\l|_{c,1/d}$; for the former it is preferable to choose
  $\widetilde{\Omega}$ only slightly larger than
  $\cup_{i\in\I}\phi_i(\Omega)$, whereas the latter is achieved by
  taking $\widetilde{\Omega}$ only slightly smaller than $\Omega'$.
\end{rem}

We now wish to consider the transfer operator $\l$
as an endomorphism of a space $A^2(\Omega)$,
and derive explicit bounds on its eigenvalues.
For this it is convenient to define, for $a,\alpha>0$,
$$  {\mathcal  E} (a,\alpha):= \all{x\in \C^\N}{
|x|_{a,\alpha}:=
\sup_{n\in\N}
  |x_n| \exp(an^{\alpha}) < \infty}\,,\quad
{\mathcal  E} (\alpha):=\bigcup_{a>0}{\mathcal  E}(a,\alpha) \,.
$$
  The following
result is from \cite{expoclass};
for completeness we give the short proof here.

\begin{lem}\label{prop2}
  Let $\alpha>0$.
If
  $A\in E(\alpha)$ then
$\lambda(A)\in {\mathcal E}(\alpha)$.
More precisely, if
$A\in E(c,\alpha)$
then
$\lambda(A) \in{\mathcal E}(c/(1+\alpha),\alpha)$, with
$
|\lam(A)|_{c/(1+\alpha),\alpha}\ \leq\ |A|_{c,\alpha}
$.
\end{lem}
\begin{proof}
  If $A\in E(c,\alpha)$ then $s_k(A)\leq \abs{A}_{c,\alpha}
  \exp(-ck^\alpha)$.  The multiplicative Weyl inequality
  \cite[3.5.1]{pietsch2} gives
\begin{equation*}
\abs{\lam_k(A)}^k \leq\prod_{l=1}^k\abs{\lam_l(A)}\leq
\prod_{l=1}^ks_l(A)\leq\prod_{l=1}^k\abs{A}_{c,\alpha}
\exp(-cl^\alpha) =\abs{A}_{a,\alpha}^k\exp(-c\sum_{l=1}^k l^\alpha)\,,
\end{equation*}
and $\sum_{l=1}^k l^\alpha\geq
\int_0^kx^{\alpha}\,dx=\frac{1}{1+\alpha}k^{\alpha+1}$,
so
$|\lam_k(A)|\leq|A|_{c,\alpha}\exp( -ck^\alpha/(1+\alpha))$.
\end{proof}

\begin{rem}
  Lemma~\ref{prop2} is sharp, in the sense that there exists an
  operator $A\in E(c,\alpha)$ such that $\lam(A)\not\in {\mathcal
    E}(b,\alpha)$ whenever $b>c/(1+\alpha)$ (see
  \cite[Proposition~2.10]{expoclass}).
\end{rem}

The following result is a detailed
version of the theorem stated in \S \ref{introsection}.

\begin{theorem}
\label{expliciteigenvaluebounds}
Let $(\Omega, \phi_i,w_i)_{i\in\I}$ be a holomorphic map-weight
system on
$\Omega\in\open$.  Let $\widetilde{\Omega}\cc\Omega$ be such that
\begin{equation*}
\cup_{i\in\I}\phi_i(\Omega) \cc \widetilde{\Omega}\cc\Omega\,,
\end{equation*}
and such that $(\Omega,\widetilde{\Omega})$ has a relative cover of
size $N$ with efficiency $\Gamma$.
Then the eigenvalue sequence $\lam(\l)$ of the corresponding transfer
operator $\l:A^2(\Omega)\to A^2(\Omega)$ satisfies
\[ \lam(\l)\in{\mathcal  E}(dc/(1+d),1/d)\quad\text{with }
|\lam(\l)|_{dc/(1+d),1/d}\leq |\l|_{c,1/d}, \] where
$c=\norm{\Gamma}{d}$, and $|\l|_{c,1/d}$ can be bounded as in
(\ref{trans:eq36b}).

In particular,
\[ |\lam_n(\l)|\leq |\l|_{c,1/d}\exp\left( -\left(\frac{dc}{1+d}
  \right)n^{1/d}\right)\quad \text{for all } n\in\N\,. \]
\end{theorem}
\begin{proof}
  This follows from Theorem \ref{mainthm} and the case $\alpha=1/d$ in
  Lemma \ref{prop2}.
\end{proof}

\section{An application: Taylor coefficients of the determinant}
\label{determinantsection}

By Theorem~\ref{mainthm}, the
transfer operator
$\l:A^2(\Omega)\to A^2(\Omega)$
for a holomorphic
map-weight system on $\Omega$
is
trace class, so we may consider the corresponding spectral determinant
$\det(I-\zeta\l)$, given for small $\zeta\in\mathbb C$
by (see e.g.~\cite[Chapter 3]{simon})
\begin{equation}
\label{expr1}
\det(I-\zeta\l)=\exp(- \sum_{n=1}^\infty a_n(\l) \zeta^n)\,,
\end{equation}
where $a_n(\l) = \frac{1}{n} \text{tr}(\l^n)$.
This formula admits a holomorphic extension to the whole
complex plane, so that $\zeta\mapsto \det(I-\zeta\l)$ becomes an
entire function. Writing
\begin{equation}
\label{expr2}
\det(I-\zeta\l)= 1+\sum_{n=1}^\infty \alpha_n(\l)\,\zeta^n\,,
\end{equation}
the Taylor coefficients $\alpha_n(\l)$ can be bounded as follows:

\begin{theorem}
\label{taylorcoeffbounds}
Let $\l$ be the transfer operator associated to the holomorphic
map-weight system $(\Omega,\phi_i,w_i)_{i\in\I}$ on $\Omega\in\open$.
If $\det(I-\zeta\l)=\sum_{n=0}^\infty \alpha_n(\l) \zeta^n$
then
\begin{equation}
|\alpha_n(\l)| \le
\left|\l\right|_{c,1/d}^n\,\exp\left(-\frac{d}{d+1}cn^{1+1/d}+\sum_{i=0}^d
\frac{d!}{(d-i)!} \frac{n^{1-i/d}}{c^i}\right)
\end{equation}
for all $n\in\N$, where $c$ and $|\l|_{c,1/d}$ can be chosen as in
Theorem~\ref{mainthm}.
\end{theorem}
\begin{proof}
  By \cite[Lemma 3.3]{simon},
  $$
  \alpha_n(\l)= \sum_{i_1<\ldots< i_n} \prod_{j=1}^n
  \lambda_{i_j}(\l)\,,
  $$
  the summation being over $n$-tuples of positive integers
  $(i_1,\ldots,i_n)$ with $i_1<\ldots<i_n$.  Now
  $$
  \sum_{i_1<\ldots <i_n} \left| \prod_{j=1}^n
    \lambda_{i_j}(\l)\right| \le \sum_{i_1<\ldots <i_n} \prod_{j=1}^n
  s_{i_j}(\l)\,,
  $$
  by \cite[Cor.~VI.2.6]{ggk}.  But $s_n(\l)\le |\l|_{c,1/d}
  \exp(-cn^{1/d})$ for all $n\in\N$, so
\begin{equation}
\label{hilbertsharpening:1}
 |\alpha_n(\l)|\leq |\l|_{c,1/d}^n \, \beta_n(c,d)\,,
\end{equation}
where $\beta_n=\beta_n(c,d)$
are the Taylor coefficients of the function
$f_{c,1/d}$ defined by
\[ f_{c,1/d}(\zeta)=\prod_{n=1}^\infty(1+\zeta\exp(-cn^{1/d}))
=\sum_{n=0}^\infty \beta_n(c,d)\zeta^n\,. \] Fried
\cite[p.~507]{fried} estimates
$
\log 1/\beta_n\ge n\log r - c^{-d} P(\log r)
$,
where $P(x):=\sum_{i=0}^{d+1}\frac{d!}{i!}x^i$.  Setting $\log r
=cn^{1/d}$ gives
\begin{equation*}
  \beta_n \le\exp\left(-cn^{1+1/d}+c^{-d}P(cn^{1/d})\right)
  =\exp\left(-\frac{d}{d+1}cn^{1+1/d}+\sum_{i=0}^d \frac{d!}{(d-i)!}
    \frac{n^{1-i/d}}{c^i}\right)\,,
\end{equation*}
and combining with (\ref{hilbertsharpening:1}) gives the required
bound on $\alpha_n(\l)$.
\end{proof}

One motivation for Theorem \ref{taylorcoeffbounds} is the possibility
of obtaining \emph{a posteriori} bounds on the eigenvalues of transfer
operators $\l:A^2(\Omega)\to A^2(\Omega)$.  In other words, for a
particular $\l$, we wish to rigorously bound the quality of
computed approximations to the eigenvalues $\lambda_i(\l)$.  In
particular cases these bounds may be sharper than the a priori
estimates of \S \ref{transferoperatorssection}.
In dimension $d=1$ such
rigorous a posteriori analysis has been performed in
\cite{jenkinsonpollicottbolyairenyi,jenkinsonpollicottjuliaklein}.
 The
bounds on $\alpha_n(\L)$ in Theorem \ref{taylorcoeffbounds} are
sharper than those of
\cite{jenkinsonpollicottbolyairenyi,jenkinsonpollicottjuliaklein}, and
valid for arbitrary $\Omega$ in arbitrary dimension $d$.

We now outline the method of a posteriori analysis based on Theorem
\ref{taylorcoeffbounds}.  Comparison of the two expressions
(\ref{expr1}) and (\ref{expr2}) for $\det(I-\zeta \l)$ yields the
identity
\begin{equation}
\label{alphaformula}
\alpha_n(\l)= \sum_{(n_1,\ldots,n_j)\atop n_1+\cdots+n_j=n}
\frac{(-1)^j}{j!}\prod_{l=1}^j a_{n_l}(\l)\,.
\end{equation}
In particular, each $\alpha_n(\l)$ is expressible in terms of
$a_1(\l),\ldots,a_n(\l)$.  The importance of this is underscored by
Ruelle's observation \cite{ruelleinventiones} that each $a_i(\l)$ can
itself be expressed in terms of fixed points (which are numerically
computable) of compositions of the maps $(\phi_i)_{i\in\I}$.
More precisely, if $\underline i :=(i_1,\ldots,i_n)\in\I^n$ then
$\phi_{\underline
  i}:=\phi_{i_n}\circ\cdots\circ\phi_{i_1}$ has a unique fixed point
$z_{\underline i}$ \cite[Lem.~1]{ruelleinventiones}.  If
$w_{\underline i}:= \prod_{j=0}^{n-1} w(z_{\sigma^j\underline i})$,
where $\sigma^j\underline i:=(i_{j+1},\ldots i_n ,i_1,\ldots, i_j)$,
Ruelle's formula is
\begin{equation}
\label{aformula}
a_n(\l)=\frac{1}{n} \text{tr}(\l^n)
= \frac{1}{n} \sum_{\underline i\in \I^n}
\frac{w_{\underline i}}{\det(I-\phi_{\underline i}'(z_{\underline i}))}\,,
\end{equation}
where $\phi_{\underline i}'$ denotes the derivative of
$\phi_{\underline i}$.

Now fix $N\in\N$ such that for all $\underline i\in \cup_{1\le n\le N}
\I^n$, the fixed point $z_{\underline i}$ can be determined
computationally
to a given numerical precision.  The Taylor coefficients
$\alpha_1(\l),\ldots,\alpha_N(\l)$ may then be computed via
(\ref{alphaformula}), (\ref{aformula}), and used to define the
polynomial function $\Delta_N(\zeta):=1+\sum_{n=1}^N
\alpha_n(\l)\zeta^n$, an approximation to
$\Delta(\zeta):=\det(I-\zeta\l)$.  If $\zeta_1, \zeta_2,\ldots$
are the zeros of $\Delta$, ordered by increasing modulus and listed
with multiplicity, then
each
$\zeta_i=\lambda_i(\l)^{-1}$.  Let $\zeta_{N,1},\ldots, \zeta_{N,N}$
denote the zeros of $\Delta_{N}$, ordered by increasing modulus and
listed with multiplicity; these zeros can be computed to a given
precision, and their reciprocals will approximate the corresponding
eigenvalues of $\l$.  In this way any eigenvalue $\lambda_i(\l)$ may
be approximated by the numerically computable values
$\zeta_{N,i}^{-1}$.
A practical issue concerns the quality of this approximation,
and it is here that the a priori bounds on the $\alpha_n(\l)$ can be
used.  The error $|\zeta_i-\zeta_{N,i}|$ may be bounded
using Rouch\'e's theorem: if $C$ is a
circle of radius $\varepsilon>0$, centred at $\zeta_{N,i}$ and
enclosing no other zero of $\Delta_{N}$, and  if it can be shown that
\begin{equation}
\label{rouche}
|\Delta_N(\zeta)-\Delta(\zeta)|<|\Delta_N(\zeta)|
\quad\text{for }\zeta\in C\,,
\end{equation}
then $\zeta_{i}$  lies inside $C$, so
$|\zeta_i-\zeta_{N,i}|<\varepsilon$.  As
$
\Delta(\zeta)-\Delta_N(\zeta) = \sum_{n=N+1}^\infty \alpha_n(\l)
\zeta^n
$,
the lefthand side of (\ref{rouche}) can be estimated in terms of
$\alpha_n(\l)$, $n>N$, which are bounded by Theorem
\ref{taylorcoeffbounds}.

\section{Appendix A: A proof of Grothendieck's Remarque 9}
\label{grothendiecksection}

In his thesis
\cite{grothendieck}, Grothendieck proved that the
eigenvalues of a bounded operator on a
quasi-complete nuclear space decrease {\it rapidly} \cite[Chap II,
\S 2, No. 4, Corollaire 3]{grothendieck}. He also noted
that this result could be improved
for certain spaces: in \cite[Chap II, \S 2, No.4, Remarque 9]{grothendieck} he
provides a sketch of a proof that shows that the eigenvalue sequence
$\lambda(L)$ of any bounded operator
$L$ on ${\mathcal  H}(\Omega)$, $\Omega\in\open$,
satisfies\footnote{Grothendieck in fact asserted that
$\lam(L)\in {\mathcal  E}(1)$, though his arguments can be modified so
as to yield $\lambda(L) \in {\mathcal  E}(1/d)$.}
$\lambda(L) \in {\mathcal  E}(1/d)$.

The results of this paper
allow us
to give a short alternative proof of Grothendieck's Remarque 9.
Let
$\set{\Omega_n}_{n\in\N}$ be a collection of members of $\open$
such that
$\Omega_n\cc\Omega_{n+1}$  for $n\in\N$,
and $\cup_{n\in\N}\Omega_n=\Omega$.
For $n\in\N$, define the seminorm
 $p_n$ on ${\mathcal  H}(\Omega)$ by
$
 p_n(f):=\sqrt{ \int_{\Omega_n}\abs{f(z)}^2\,dV(z)}
$
(note that
$p_n$ gives the norm on $A^2(\Omega_n)$).
Then $\{p_n\}$ forms a directed system of seminorms which
turns
${\mathcal  H}(\Omega)$ into a Fr\'echet space and which, by
Lemma~\ref{berg:lem1}, coincides with the usual topology of uniform convergence
on compact subsets of $\Omega$.
Moreover, since each
identification
$A^2(\Omega_{n+1})\rightarrow A^2(\Omega_n)$ is nuclear by
Theorem~\ref{embeddingtheorem}, the space ${\mathcal  H}(\Omega)$ is nuclear.

Recall that a subset $S$ of a topological vector space $E$ is {\it
bounded} if for each neighbourhood $U$ of $0$, we have $S \subset
\alpha U$ for some $\alpha > 0$. A linear operator $L:E\to E$ is
{\it bounded} if it takes a neighbourhood of
  0 into a bounded set.
We are now able to prove the following.

\begin{theorem}\label{grothendiecktheorem}[Grothendieck]
Suppose $\Omega\in \open$, and
$L:{\mathcal  H}(\Omega)\to {\mathcal  H}(\Omega)$
is a bounded linear operator. Then:
\begin{itemize}
\item[(i)]
There exists a sequence $\set{s_k}$ of positive numbers
belonging to ${\mathcal  E}(1/d)$,
an equicontinuous sequence $\set{f_k'}$ in
the topological dual ${\mathcal H}(\Omega)'$
of ${\mathcal  H}(\Omega)$, and a bounded sequence
$\set{f_k}$ in ${\mathcal  H}(\Omega)$, such that
$L$ can be written
\[ Lf=\sum_{k} s_k\dual{f}{f_k'}f_k \quad
\text{for all $f\in {\mathcal  H}(\Omega)$}\,.\]
Here, $\dual{f}{f'}$ denotes the evaluation of $f'\in {\mathcal
  H}(\Omega)'$ at $f$.
\item[(ii)] $\lam(L)\in {\mathcal  E}(1/d)$.
\end{itemize}
\end{theorem}
\begin{proof} The two assertions will follow from a factorisation of
  $L$, which we shall first derive.
  Since $L$ is bounded, there exists $n_0\in\N$ such that for every
  $n\in\N$, there is a constant $M_n$ satisfying
$p_n(L f)\leq M_np_{n_0}(f)$ for all
$f\in {\mathcal  H}(\Omega)$.
Fixing $m>n_0$,  let
$ J_1:{\mathcal  H}(\Omega)\rightarrow A^2(\Omega_m)$ and
 $J_2:A^2(\Omega_m)\rightarrow A^2(\Omega_{n_0})$
denote canonical identifications.
Clearly,
$J_1$ and $J_2$ are continuous. Let $\overline{J_2J_1{\mathcal
    H}(\Omega)}$ be the closure of $J_2J_1{\mathcal  H}(\Omega)$ in the
Hilbert space $A^2(\Omega_{n_0})$ and let $P: A^2(\Omega_{n_0})
\rightarrow \overline{J_2J_1{\mathcal  H}(\Omega)}$ be the corresponding
orthogonal projection. Then the linear map
$f \in
J_2J_1{\mathcal  H}(\Omega) \mapsto Lf \in {\mathcal  H}(\Omega)
 $
is well-defined and bounded, and therefore extends to a bounded linear
map
$\widetilde L : \overline{J_2J_1{\mathcal  H}(\Omega)} \to {\mathcal  H}(\Omega)$.
The operator $L$ therefore admits the factorisation
\begin{equation}
\label{boundedfactorisation}
L=\widetilde{L}PJ_2J_1\,.
\end{equation}
To prove (i), note that $J_2\in E(1/d)$ by
Theorem~\ref{embeddingtheorem}, so we have the Schmidt representation
$J_2f=\sum_ks_k(J_2)\inn{f}{a_k}_mb_k$,
where $\set{a_k}$ and $\set{b_k}$ are orthonormal
systems in $A^2(\Omega_m)$ and $A^2(\Omega_{n_0})$ respectively and
$\inn{\cdot}{\cdot}_m$ denotes the inner product in $A^2(\Omega_m)$.
Since $\widetilde{L}P$ is continuous,
\begin{equation*}
Lf = \widetilde{L}PJ_2J_1f = \sum_k
s_k(J_2)\inn{J_1f}{a_k}_m\widetilde{L}Pb_k\,,
\end{equation*}
which can be written as
\begin{equation}
\label{representation:e1}
Lf = \sum_k s_k(J_2)\dual{f}{J_1'a_k'}\widetilde{L}Pb_k\,,
\end{equation}
where $J_1'$ denotes the adjoint of $J_1$, and $a_k'$ the image of
$a_k$ under the canonical isomorphism of $A^2(\Omega_m)$ and its dual.
In order to see that the representation~(\ref{representation:e1}) has
the desired properties, we note that $\set{\widetilde{L}Pb_k}$ is
bounded, since it is the continuous image of a bounded set.
Furthermore, $\set{J_1'a_k'}$ is equicontinuous in the dual of ${\mathcal
  H}(\Omega)$, since
\[ \abs{\dual{f}{J_1'a_k'}}=\abs{\inn{J_1f}{a_k}_m}\leq
p_m(J_1f)p_m(a_k)\leq p_m(f)\,.\] Therefore (i) is proved.

To prove (ii) we again use the
factorisation~(\ref{boundedfactorisation}). By Pietsch's principle of
related operators (see \cite[Satz 2]{pietschfredholm}),
$\lam(L)
=  \lam(\widetilde{L}PJ_2J_1)=\lam(J_1\widetilde{L}P J_2)$.
But $J_1\widetilde{L}P :A^2(\Omega_{n_0})\to A^2(\Omega_m)$ is a
bounded operator between Hilbert spaces, and $J_2\in E(1/d)$ by
Theorem~\ref{embeddingtheorem}, so
$J_1\widetilde{L}PJ_2\in E(1/d)$ by Lemma~\ref{expoprop2},
hence $\lam(J_1\widetilde{L}PJ_2)\in {\mathcal  E}(1/d)$
by Lemma~\ref{prop2}, and (ii) follows.
\end{proof}

\begin{rem}
  In our approach, assertion (ii) of Theorem \ref{grothendiecktheorem}
  follows by combining Theorem \ref{embeddingtheorem} with Weyl's
  multiplicative inequality, whereas Grothendieck suggests to derive
  (ii) from (i) by considering the growth of the determinant
  $\det(I-\zeta L)$ at infinity and using Jensen's theorem to
  determine bounds on the distribution of its zeros.  A more detailed
  analysis of this circle of ideas will be presented in the following
  \S \ref{determinantbasedsection}.
\end{rem}

\section{Appendix B: Eigenvalue estimates via the determinant}
\label{determinantbasedsection}

Given a transfer operator $\l$ associated to a holomorphic map-weight
system on $\Omega\in\open$, we have shown
(Theorem~\ref{mainthm}) how to find explicit constants
$a,A>0$ such that
\begin{equation}
\label{notinfried}
s_n(\l)\le
A\exp(-an^{1/d})
\quad\text{for all }n\in\N
\,,
\end{equation}
and used this (Theorem~\ref{expliciteigenvaluebounds}) to
find explicit $b,B>0$ for which
\begin{equation}
\label{evaluesapp2}
|\lambda_n(\l)|\le B\exp(-bn^{1/d})
\quad\text{for all }n\in\N
\,.
\end{equation}
The purpose of this appendix is to outline an alternative, less
direct, method of obtaining eigenvalue bounds analogous to
(\ref{evaluesapp2}), again starting from the singular value estimate
(\ref{notinfried}).  This approach is based on an analysis of the
growth of the determinant $\det(I-\zeta\l)$, and was originally
suggested by Grothendieck in \cite[Chap. II, \S2, No. 4, Remarque
9]{grothendieck}. Further details of this strategy were given by Fried
\cite{fried}, and we shall offer some commentary on Fried's analysis,
in particular his Lemma 6, adapted slightly to our Hilbert space
setting.

A bound of the type (\ref{notinfried}) is not proved in \cite{fried},
though does appear to be tacitly assumed \cite[p.~506, line 8]{fried},
on the basis of a suggested correction of \cite[II, Remarque 9,
p.~62--4]{grothendieck} (see \cite[p.~506, line 3]{fried}, and our
comments in Sections \ref{introsection} and
\ref{grothendiecksection}).
With the singular value estimate (\ref{notinfried}) in hand, it is
possible to analyse the growth properties of
the
function
$\zeta\mapsto\det(I-\zeta\l)$, which
is entire because $\l$ is trace class
   (see \S \ref{determinantsection}).
This is the content of \cite[Lemma
6]{fried}, which we now review, incorporating some refinements
available in the Hilbert space setting.  We start by writing
\[ \det(I-\zeta\l)=\sum_{n=0}^\infty \alpha_n(\l) \zeta^n \,. \]
As in Theorem \ref{taylorcoeffbounds} we use the formula
$$
\alpha_n(\l)= \sum_{i_1<\ldots< i_n} \prod_{j=1}^n
\lambda_{i_j}(\l)\,,
$$
and the inequality
$$
\sum_{i_1<\ldots <i_n} \left| \prod_{j=1}^n
  \lambda_{i_j}(\l)\right| \le \sum_{i_1<\ldots <i_n} \prod_{j=1}^n
s_{i_j}(\l)\,,
$$
to deduce that
\begin{equation}
\label{hilbertsharpening}
 |\alpha_n(\l)|\leq A^n\beta_n(a,d)\,,
\end{equation}
where $\beta_n(a,d)$ are the Taylor coefficients of the function
$f_{a,1/d}$ defined by
\[ f_{a,1/d}(\zeta)=\prod_{n=1}^\infty(1+\zeta\exp(-an^{1/d}))
=\sum_{n=0}^\infty \beta_n(a,d)\zeta^n\,. \] Note that
(\ref{hilbertsharpening}) is sharper than the corresponding estimate
in \cite[p.~506]{fried}, which contains an extra factor $n^{n/2}$.
Following Fried, the coefficients $\beta_n=\beta_n(a,d)$ can be
estimated, using Cauchy's theorem, by $\beta_n\leq r^{-n}M(r)$, where
$M(r)$ is the maximum modulus of $f_{a,1/d}(\zeta)$ on $|\zeta|=r$.
Using either the asymptotics
\[ \log f_{a,1/d}(r)\sim a^{-d} \frac{1}{d+1}(\log r)^{1+1/d}
\quad\text{as }r\to\infty
 \]
 in
\cite[Proof of Proposition~3.1~(i)]{expoclass}, or Fried's calculation that
 $
 \log 1/\beta_n\ge n\log r - a^{-d}P(\log r)\,,
 $
 where $P(x):=\sum_{j=0}^{d+1}\frac{d!}{j!}x^j$, we see that for
 any $\delta_0>1$,
\[ \log 1/\beta_n \geq n\log r - \delta_0 a^{-d} \frac{1}{d+1}(\log
r)^{1+1/d}, \] for $r$ sufficiently large.  Choosing $\log r=an^{1/d}$
gives
$\log 1/\beta_n \geq \delta_1 a \frac{d}{d+1}n^{1+1/d}$
for $n$ sufficiently large, where $\delta_1=1-(\delta_0-1)/d$.
Therefore there exists $K>0$, depending on $\delta_1$, such that
\[ |\alpha_n(\l)| \leq K A^n \exp\left( -\delta_1 a
  \frac{d}{d+1}n^{1+1/d} \right)\quad
\text{for all }n\in\N\,.  \]
Thus
if $
g(r):= \sum_{n=1}^\infty r^n \exp\left( -\delta_1 a
  \frac{d}{d+1}n^{1+1/d} \right)
$ then
\[  |\det(I-\zeta\l)|\leq 1+ K\sum_{n=1}^\infty |\zeta|^n A^n
\exp\left( -\delta_1 a \frac{d}{d+1}n^{1+1/d} \right)=1+K
g(A|\zeta|)\,. \]

To estimate the growth of $g$,
define\footnote{Alternatively one could proceed as in \cite{fried}, but
  the method there is a little less sharp.}
$\mu(r):=\max_{1\leq n\leq \infty}  r^n \exp\left( -\delta_1 a
  \frac{d}{d+1}n^{1+1/d}\right)$.
This maximal term can be calculated explicitly using calculus (see
\cite[Proof of Proposition~3.1~(ii)]{expoclass}), and we obtain
\[ \log \mu(r)\sim (\delta_1a)^{-d}\frac{1}{d+1} (\log r)^{1+d}
\quad\text{as }r\to\infty \,. \]

But $g$ is an entire function of finite order, so
$\log \mu(r)\sim \log g(r)$ as $r\to\infty$
(see e.g.~\cite[Problem 54]{polyaszego}),
hence
$\log g(r)\sim (\delta_1a)^{-d}\frac{1}{d+1} (\log r)^{1+d}$
as $r\to\infty$.
Therefore, for $|\zeta|$ sufficiently large and
$\delta_2\geq \delta_1^{-d}$,
\begin{equation}
\label{logdetgrowthestimate}
 \log |\det (1-\zeta\l)| \leq \delta_2 a^{-d}
\frac{1}{d+1}(\log |\zeta|A)
^{1+d} \,.
\end{equation}

The bound (\ref{logdetgrowthestimate}) allows us to
estimate the speed with which the zeros of $\det(1-\zeta\l)$
tend to infinity.  Specifically, if $n(r)$ denotes the number of
zeros of $\det(1-\zeta\l)$ in the disk of radius $r$ centred at $0$,
and $N(r):=\int_0^rt^{-1}n(t)\,dt$, Jensen's theorem (see
e.g.~\cite[p.~2]{boasbook})
gives
\begin{equation}
\label{jensenconsequence}
 N(r)\leq \delta_2 a^{-d}\frac{1}{d+1}(\log rA)
^{1+d}
\end{equation}
for $r$ sufficiently large.
We now require the following lemma:

\begin{lem}
\label{appendixblemma}
If $N(r)\leq K(\log r)^{1+d}$ for some positive real number $d$, then
\[ n(r)\leq K \frac{(1+d)^{1+d}}{d^d}(\log r)^d. \]
\end{lem}
\begin{proof}
  If $p>1$ then
$(p-1) n(r)\log r=n(r)\int_r^{r^p}t^{-1}\,dt\leq
\int_r^{r^p}t^{-1}n(t)\,dt\leq N(r^p)$, so
$$
n(r)\leq \frac{ N(r^p)}{(p-1)\log r}\leq \frac{ K p^{1+d}(\log
  r)^{1+d}}{(p-1)\log r}\,.
$$
The assertion follows by choosing
$p=1+1/d$.
\end{proof}

Combining (\ref{jensenconsequence}) and Lemma \ref{appendixblemma}
gives
$n(r)\leq  \delta_3 a^{-d}\left (\frac{1+d}{d}\right )^d(\log rA)^d$
for $r$ sufficiently large.  But the zeros of $\det(I-\zeta\l)$ are
precisely the numbers
$\lambda_1(\l)^{-1}$,
$\lambda_2(\l)^{-1},\ldots$,
ordered by modulus, so for $n$
sufficiently large,
$n\leq \delta_3a^{-d} \left(\frac{1+d}{d}\right)^d(\log
A|\lambda_n(\l)|^{-1})^d$,
and finally we deduce the required
eigenvalue bound
\begin{equation}
\label{alternativeevaluebound}
 |\lambda_n(\l)|\leq A \exp\left( -\delta_3^{-1/d} a
\frac{d}{1+d}n^{1/d}\right)\quad\text{for }n\text{ sufficiently large.}
\end{equation}

Since $\delta_3$ can be chosen arbitrarily close to 1,
(\ref{alternativeevaluebound}) can be made arbitrarily close to the
bound of Lemma \ref{prop2}.  Note, however, that
(\ref{alternativeevaluebound}) only holds for $n\ge N$, for some
unknown $N$, whereas the bound of Theorem
\ref{expliciteigenvaluebounds} is valid for all $n\in\N$.

\end{document}